\documentclass[a4,12pt,leqno]{amsart}
\usepackage{amsmath}
\pdfoutput=1
\usepackage{amsfonts}
\usepackage{amssymb}
\usepackage{mathrsfs}
\usepackage{sidecap}
\usepackage{float}
\usepackage{graphicx}
\restylefloat{figure}

\swapnumbers \theoremstyle{plain}
\newtheorem{thm}[equation]{Theorem}
\newtheorem{lemma}[equation]{Lemma}

\theoremstyle{definition}

\newtheorem{conj}[equation]{Conjecture}
\newtheorem{exmp}[equation]{Example}
\newtheorem{notation}[equation]{Notation}
\theoremstyle{remark}
\newtheorem{rem}[equation]{Remark}

\newtheorem{algorithm}[equation]{Algorithm}
\numberwithin{equation}{section}

\begin{document}

\title[Algorithm for centered Hausdorff measure]{An algorithm for computing the centered Hausdorff measures of
self-similar sets}

\author{Marta Llorente}
\address{Marta Llorente: Departamento de
 An\'{a}lisis Econ\'{o}mico: Econom\'{i}a Cuantitativa\\Universidad Aut\`{o}noma de Madrid, Campus de Cantoblanco, 28049 Madrid \\  Spain\\ }
\email{m.llorente@uam.es }

\author{Manuel Mor\'{a}n}
\address{Manuel Mor\'{a}n: Departamento An\'{a}lisis Econ\'{o}mico I\\
Universidad Complutense de Madrid \\
Campus de Somosaguas, 28223 Madrid, Spain\\}
\email{mmoranca@ccee.ucm.es }
\thanks{This research was supported by the Ministerio de Educaci\'{o}n y Ciencia, research project MTM2006-02372.}
\date{}

\begin{abstract}
We provide an algorithm for computing the centered Hausdorff measures
of self-similar sets satisfying the strong separation
condition. We prove the convergence of the algorithm and test
its utility on some examples.
\end{abstract}

\maketitle
\section{Introduction\protect\bigskip}

We present an algorithm that takes as input a list
$\Psi=\left\{ f_{1},f_{2},...,f_{m}\right\} $ of contracting similitudes in $%
\mathbb{R}^{n}$ satisfying the strong separation condition (see
section \ref{Section 2})
and gives as output an estimate of the
\textit{$s$-dimensional centered Hausdorff measure}, $C^{s}(E)$, of the
self-similar set $E$ generated by $\Psi $. Here $s$ is both the similarity and Hausdorff dimension of $%
E$ and it can be computed from the contracting factors of the
similitudes (see (\ref{dim})). To our knowledge this is the first
attempt at automatic computation of the exact value of a metric
measure (e.g. Hausdorff, packing, spherical, centered, ...), a topic
which has generated a
considerable quantity of research (\cite{ayer}, \cite{daitian}, \cite%
{jiazhouzhu}, \cite{jiazhouzhuluo}, \cite{meifeng&},\cite{zhiweizuoling},
\cite{zhou}, \cite{zhoufeng} \cite{zhufeng1}, \cite{zhuzhou}, \cite{zhuzhou0}%
, \cite{zhuzhou1}, etc.)


The centered Hausdorff measure is a variant of the Hausdorff
measure. These measures differ mainly in the natures of the coverings used in their definitions.
In the case of the centered Hausdorff
measure we consider only covers by closed balls $B(x_{i})$
 centered at points $x_{i}$ in the given set.

The standard definition of $C^{s}$ on subsets of $\mathbb{R}^{n}$
consists of two steps. Given $A\subset \mathbb{R}^{n}$, we first
compute the premeasure $C_{0}^{s}$ as
\begin{equation}\label{paso1}
C_{0}^{s}A=\lim_{\delta \rightarrow 0}\inf \left\{ \sum_{i\in \mathbb{N}%
}\left\vert B_{i}(x_{i})\right\vert ^{s}:A\subset \cup B_{i}\text{ and }%
\left\vert B_{i}\right\vert <\delta ,x_{i}\in A\text{ for all }i\right\} .
\end{equation}
Because the suppression of good candidates for the $x_{i}$'s may cause
an increase of the infimum, this premeasure is not monotone,
although it is $\sigma-$additive (see \cite{tricot0}). In order
to avoid this difficulty we define
\begin{equation}\label{paso2}
C^{s}A=\sup \left\{ C_{0}^{s}B:B\subset A,B\text{ closed}\right\}.
\end{equation}%
The set function $C^{s}$ so obtained is a metric measure. It
turns out that the centered Hausdorff measure is bounded by constant
multiples of the ordinary Hausdorff measure (see
\cite{sainttricot}). More precisely
\begin{eqnarray*}
2^{-s}C^{s}(E)\le H^{s}(E)\le C^{s}(E)
\end{eqnarray*}  and so
the centered Hausdorff dimension and the ordinary Hausdorff dimension coincide.
In particular, for the self-similar
set $E,$ we have that $0<C^{s}(E)<\infty $.
Thus $C^{s}$ is a nice measure, but there remains the question of how to compute $C^{s}A$ for
some subset $A\subset E$. A simplification comes from the
following observation: for any Borel set $A \subset
\mathbb{R}^{n}$,
 $$C^{s}A=C^{s}E\mu A,$$ where is $\mu $ the so
called \textit{natural }or \textit{empirical probability measure
on $E$.}
Therefore, the problem of computing $ C^{s} $ on $E$ reduces to
computing $C^{s}E$. Observe that the $\mu-$measure of any open
subset $A\subset $ $E$ with boundary $\partial A$ having null $\mu
-$measure (and in particular of any open ball) can be easily
obtained with arbitrary accuracy and, hence, the $\mu$-measures of
compact subsets of $E$ with $\mu-$null boundaries can also be
computed through their complements. This gives a vast class
$\mathcal{C}$ of Borel sets with computable $\mu -$measure. In
turn, the $C^{s}-$measure of any $C^{s}-$measurable set can be approximated
with arbitrary accuracy by the $C^{s}-$measures of closed sets (see
\textbf{\cite{falconer} }Theorem 1.6 b).

How, then, to compute $C^{s}E$? Given the definitions, the task
seems out of reach. The first obstruction comes from the need of
the second step \ref{paso2} in the definition of $C^{s}$.
However, in \cite{llorentemoran1} it is proved
that, for any subset $A$ of a self-similar set $E$ as above the
measure and the premeasure coincide:

\begin{thm}[Theorem 3 in \cite{llorentemoran1}]
Let $A$ be either a closed or an open subset of a self-similar set
$E$ satisfying the open set condition. Then
$C_{0}^{s}A=\nolinebreak C^{s}A$.
\end{thm}

With this result available, $C^{s}$ seems easier to compute than the
Hausdorff or spherical Hausdorff measure. Namely, the differences
between these three measures are
that, for the Hausdorff measure, one
optimizes among coverings by arbitrary convex sets; for the
spherical Hausdorff measure one uses arbitrary balls; and for
$C^s$ one uses coverings by balls with centers in $E$.
For the Hausdorff and the spherical Hausdorff measures,
the classes of available coverings are larger, and hence it is more
difficult to find optimal coverings.

The second step which permits the computation of $C^{s}$ was also
given in \cite{llorentemoran1}. There it is proved that computing
$C^{s}$ is equivalent to finding a centered ball with optimal
inverse density:
\begin{thm}[Theorem 5 in \cite{llorentemoran1}]
\label{tcentered}
Suppose the invariant set $E$ of the system $\Psi$ satisfies the open set condition, with $\dim _{H}E=s$ and $|E|=R$, and let $%
\mu $ be the normalized Hausdorff measure on $E$. Then
\begin{equation}
C^{s}E=\inf \left\{ \frac{(2d)^{s}}{\mu (B(x,d))}:x\in
E,d>0\right\} =:D_{C}^{-1} . \label{cenosc}
\end{equation}
Moreover, if $\ \Psi $\ satisfies the SSC then
\begin{equation}
C^{s}E=\min \left\{ \frac{(2d)^{s}}{\mu (B(x,d))}:x\in E\text{ \
and }c\leq d\leq R\right\},  \label{censsc}
\end{equation}%
where $c:=\min_{i,j\in M,i\neq j}(f_{i}E,f_{j}E)$ and $R:=|E|.$
\end{thm}
From now on, $B(x,d)$ denotes the closed radius $d$ ball centered at $x
\in \mathbb{R}^n$.

The statement (\ref{censsc}) is crucial in
our approach since it gives that if $B(x,d)$ is a ball of maximal
inverse density, then $C^s(E)=\frac{(2d)^{s}}{\mu(B(x,d))}$.
Therefore, to obtain the value of $C^s(E)$ we need only to find
an optimal ball and compute its density. This is precisely how
the algorithm proceeds. It searches for balls that maximize
the density function $h_{s}(x,d)=\frac{(2d)^{s}}{\mu(B(x,d))}$.
Actually, by \textit{the} \textit{self-similar tiling principle
}(see\textit{\ \cite{moran1}}), we know that finding an optimal
ball is equivalent to finding an optimal covering. This is so
because we can get an optimal covering tiling the set $E$ with
balls of optimal density.

$E$ may be tiled, without loss of $\mu$-measure, by tiles similar
to a given tile $B$. By similar we mean that the tile is an image
of $B$ under a composition of similitudes in $\Psi $. The only
condition to be imposed on $B$ is that it be closed and have $\mu
B>0.$

Our algorithm computes $C^{s}E$ and provides an
(approximate) ball of maximal density, together with an optimal
covering of $E$ by balls centered at $E$.

Now we describe the main steps of the algorithm. This is done
rigorously in section~\ref{description}. Recall that the goal is
to find the maximal value of
$h_{s}(x,d)=\frac{(2d)^{s}}{\mu(B(x,d))}$ for $x\in E$ and $d \in
[c,R]$ (see (\ref{censsc})). At step $k$, the set $E$ is replaced
by a finite set of points $\{A_{k}\}$ such that $\overline{\cup
A_k}=E$ and the measure $\mu$ is replaced by a discrete
probability measure $\mu_{k}$ supported on $A_k$ and
 converging weakly to $\mu$
(see (\ref{muigulapto}), (\ref{mukdif}) and Lemma~\ref{keylema}
(iv)).
The objective now is to find the
maximum of the discrete density function
$$h_{k}(x,d(x,y))=\frac{(2d(x,y))^{s}}{\mu_k(B(x,d(x,y))} \quad \textrm{with}\  x, y \in A_{k}.$$
Here $d(\ ,\ )$ stands for the Euclidean distance. For each $x\in
A_k$  the algorithm searches for the maximal value of
$h_{k}(x,d(x,y))$ for $y \in A_k$. Once this has been found for every
$x\in A_k$, the algorithm finds the maximum of these values with respect to $x$.
Thus we only need to compute exactly
$\mu_k(B(x,d(x,y))$ for every $y \in A_{k}$. To this end, the
points $y\in A_{k}$ are listed in order of increasing distance to $x$, and
thus  the points preceding a given point $y$ in the list always
belong to the ball $B(x,d(x,y)).$ It is not hard to see that
the exact value of $\mu_k (B(x,d(x,y))$ is obtained from the place
of $y$ in this list (see (\ref{mukigualball}) for the homogeneous
case and (\ref{mudifball}) for the general case). It remains to
show that in the limit $k \nearrow \infty$ this process  converges
to $C^s(E)$. This is done in section~\ref{convergsec}. The
convergence is shown in two steps. First, by means of the Markov
operator $M$ associated to the set $E$ (see
section~\ref{preliminaries} for notation and definitions), the
measures $\mu_k$ are shown to converge weakly to the invariant
measure $\mu$. The basic properties of these measures yield a
sequence of pairs of points $(x_k,y_k) \in A_k\times A_k$ such
that $h_{k}(x_k,d(x_k,y_k)) \to C^s(E)$. However, there is no
reason why these $(x_k,y_k)$ should optimize $h_{k}$.
Nonetheless, we are able to show that
this holds asymptotically, which is enough for our purposes. An interesting technical point is that in the proofs
an essential role is played by a result of Mattila \cite{mattila0} implying that $\mu(\partial B(x,d))=0$.

In section\nolinebreak\ \ref{Ejemplos} we apply the algorithm
to treat several sets whose centered Hausdorff measures were
available in the literature. It is remarkable that in all these
cases the optimal value (and also optimal ball and covering) is
attained at an early iteration. The algorithm also yields
conjectural values (which, in many cases, are upper bounds) for
sets whose measure is unknown. Research is in progress to
explore the rate of convergence and show that the method yields
the precise values of $C^s(E)$. Preliminary results seem to
indicate that, for self-similar sets with less than four
similarities, four decimal digits of accuracy are attainable by
personal computers without any serious effort to optimize the code's design.

\section{Preliminaries\label{Section 2}}\label{preliminaries}

Let $\Psi =\left\{ f_{1},f_{2},...,f_{m}\right\} $ be a list of contracting
similitudes on $\mathbb{R}^{n}$, $|f_{i}(x)-f_{i}(y)|=r_{i}|x-y|$ where $%
0<r_{i}<1$. The unique non-empty compact set satisfying
\begin{equation*}
E=S\Psi (E),
\end{equation*}%
where $S\Psi (X)=\bigcup\limits_{i\in M}f_{i}(X)$, is called the \textit{%
self-similar set associated to }$\Psi $. Sometimes we shall refer to $E$ as
the \textit{attractor} or \textit{invariant set} of the \textit{iterated
function system }(IFS)$\ \Psi =\left\{ f_{1},f_{2},...,f_{m}\right\} $.

We shall use the following notation. Let $M:=\{1,...,m\}$ and
\begin{equation*}
M^{k}=\{\mathbf{i}_{k}=(i_{1},...,i_{k}):i_{j}\in M\quad \forall j=1,...m\}.
\end{equation*}%
For $\mathbf{i}_{k}=i_{1}i_{2}..i_{k}\in M^{k},$ we write
\begin{eqnarray*}
f_{\mathbf{i}_{k}} &=&f_{i_{1}}\circ f_{i_{2}}\circ ...\circ f_{i_{k}}\text{,%
} \\
r_{\mathbf{i}_{k}} &=&r_{i_{1}}r_{i_{2}}...r_{i_{k}}\text{,}
\end{eqnarray*}%
and for $A\subset \mathbb{R}^{n}$, we write
\begin{equation*}
A_{\mathbf{i}_{k}}=f_{\mathbf{i}_{k}}(A).
\end{equation*}%
We shall refer to the sets $E_{\mathbf{i}_{k}}=f_{\mathbf{i}_{k}}(E)$ as the
\textit{cylinder sets of generation} $k$.

Throughout the paper we shall assume that the system $\Psi $ satisfies
the strong separation condition (SSC). That is, $f_{i}(E)\cap
f_{j}(E)=\emptyset$ for all $i\neq j,$ where $i,j\in \{1,...,m\}$. The
Hausdorff dimension of $E$, $\dim _{H}E$, is given by the unique real number
$s$ such that
\begin{equation}
\sum_{i=1}^{m}r_{i}^{s}=1 .  \label{dim}
\end{equation}%
Moreover, the Hausdorff measure of $E$ is finite and positive.

We shall denote by $\mu $ the \textit{natural probability measure}, or
\textit{normalized Hausdorff} \textit{measure}, defined on the ring of
cylinder sets by%
\begin{equation}
\mu (E_{\mathbf{i}})=r_{\mathbf{i}}^{s},  \label{mucil}
\end{equation}%
and then extended to Borel subsets of $E$. The measure $\mu $ is regular in
the sense that there are positive numbers $a$ and $b$ such that%
\begin{equation}
ar^{s}\leq \mu (B(x,r))\leq br^{s}\text{\quad for }x\in E\text{ and }0<r\leq
1.  \label{regularity}
\end{equation}

Let $\mathcal{P}(%
\mathbb{R}
^{n})$ be the space of probability measures on $%
\mathbb{R}
^{n}$. The well known fact that the measure $\mu$ is the unique invariant
measure for the \textit{Markov operator} (see, for example, \cite{Barnsley}%
), plays an important role in our proofs. Let
$\mathbf{M}:\mathcal{P}(\mathbb{R}^{n})\rightarrow \mathcal{P}(%
\mathbb{R}
^{n})$
be the \textit{Markov operator} associated with the IFS $\Psi$ with probabilities $%
\{r_{1}^{s},...,r_{m}^{s}\}$,
\begin{equation}
\mathbf{M}(\nu )=r_{1}^{s}\nu \circ f_{1}^{-1}+...+r_{m}^{s}\nu
\circ f_{m}^{-1}\quad \forall \nu \in \mathcal{P}(\mathbb{R}^{n}).
\label{markvopdef}
\end{equation}%
Then $\mu $ is the unique probability measure satisfying
\begin{equation*}
\mathbf{M}(\mu )=\mu .
\end{equation*}

\section{Description of the algorithm\label{description}}

In this section we introduce an algorithm to compute the centered Hausdorff
measure for self-similar sets satisfying the SSC.

Given $A\subset \mathbb{R}^{n}$, $|A|$ stands for the diameter of
$A$.

In \cite{llorentemoran1} it is proved that if $E$ is a self-similar set
satisfying the SSC, then
\begin{equation}
C^{s}E=\min \left\{ \frac{(2d)^{s}}{\mu (B(x,d))}:x\in E\text{ \
and }c\leq d\leq R\right\}  \label{formulacentred}
\end{equation}%
where $c:=\min_{i,j\in M,i\neq j}dist(f_{i}E,f_{j}E)$ and $R:=|E|$
(see Theorem~\ref{tcentered}).

Our method depends strongly on \eqref{formulacentred} as, to find the
value of $C^{s}E$, we construct an algorithm for minimizing the
value of
\begin{equation*}
h_{s}(x,d):=\frac{(2d)^{s}}{\mu (B(x,d))}
\end{equation*}%
when $x\in E$\ and $c\leq d \leq R$.

The idea is to construct a sequence $\left\{ A_{k}\right\} $\ of countable
sets and a sequence $\{\mu _{k}\}$ of discrete measures supported on the $A_{k}$
such that the $\mu _{k}$ converge weakly to $\mu $ and $\overline{\cup
_{k=1}^{\infty }A_{k}}=E$, where $\overline{A}$ stands for the closure of $A$%
. This allows us to construct another\ sequence $\left\{ \tilde{m}%
_{k}\right\} $ converging to $C^{s}E$ by choosing on the $k$th step a
pair $(\tilde{x}_{k},\tilde{y}_{k})$ $\in $ $A_{k}\times A_{k}$ satisfying%
\begin{equation*}
\tilde{m}_{k}:=h_{k}(\tilde{x}_{k},dist(\tilde{x}_{k},\tilde{y}_{k}))=\min
\{h_{k}(x_{k},dist(x_{k},y_{k})):(x_{k},y_{k}) \in A_{k}\times
A_{k}\},
\end{equation*}%
where $h_{k}(x,d):=\frac{(2d)^{s}}{\mu _{k}(B(x,d))}$.

We describe first the algorithm for self-similar sets where all the
contraction ratios coincide, as this case illustrates better the central idea
of the construction. After this we shall explain the modifications needed to
treat the case of unequal contraction ratios. Observe that if $%
r_{i}=r_{j}:=r$ for all $i\neq j$, the \textit{invariant measure }$\mu $
satisfies that%
\begin{equation}
\mu (E_{\mathbf{i}_{k}})=r^{ks}=\frac{1}{m^{k}}\text{\quad }\forall \text{ }%
\mathbf{i}_{k}\in M^{k}.  \label{mucilig}
\end{equation}

\begin{algorithm}\label{case1}(Homogeneous case: $r_{i}=r_{j}:=r$ \ $\forall i\neq j,$
$i,j=1,...,m$)

\begin{enumerate}
\item[1.] \textbf{Construction of }$\mathbf{A}_{k}$. Let \
$A_{1}=\{x_{i}\in \mathbb{R}^{n}:f_{i}(x_{i})=x_{i},$
$i=1,...,m\}$ be the set of the fixed points
of the similitudes in $\Psi $.
For $k\in \mathbb{N}^{+}$, let $A_{k}=S\Psi (A_{k-1})$ be the set
of $m^{k}$ points obtained by
applying $S\Psi (x)=\bigcup\limits_{i\in M}f_{i}(x)$ to each of the $%
m^{k-1} $ points of $A_{k-1}$.

\item[2.] \textbf{Construction of }$\mathbf{\mu }_{k}$. For all $k
\in \mathbb{N}^{+}$, set
\begin{equation}
\mu _{k}(x)=\frac{1}{m^{k}}\text{\quad }\forall x\in A_{k.}
\label{muigulapto}
\end{equation}%
Thus,
\begin{equation*}
\mu _{k}=\frac{1}{m^{k}}\sum_{i=1}^{m^{k}}\delta _{x_{i}}
\end{equation*}%
is a probability measure with $spt(\mu _{k})=A_{k}=\{x_{1},...,x_{m^{k}}\}$.

\item[3.] \textbf{Construction of }$\mathbf{\tilde{m}}_{k}$

\begin{itemize}
\item[3.1] Given $x_{l}\in A_{k}$, compute the $m^{k}$ distances $%
dist(x_{l},x)$ for every $x\in A_{k.}$.

\item[3.2] Arrange the distances in increasing order.

\item[3.3] Let $d_{1}\leq d_{2}\leq ....\leq d_{m^{k}}$ be the list of
ordered distances. Since for each $j=1,...,m^{k}$, $B(x_{l},d_{j})$ contains
$j+t$ points of $A_{k}$, where $t$ is the cardinality of $%
A_{k}\cap \partial $ $B(x_{l},d_{j})$, i.e. $d_{l}=d_{l+1}=...=d_{l+t}\neq
d_{l+t+1}$, there holds
\begin{equation}
\mu _{k}(B(x_{l},d_{j}))=\frac{j+t}{m^{k}}.  \label{mukigualball}
\end{equation}

\item[3.4] Let $\mathbf{i}_{k}(l)=(i_{1}(l),...,i_{k}(l))\in M^{k}$ be the
unique sequence of length $k$ such that $x_{l}=f_{\mathbf{i}_{k}(l)}(x)$
for some $x\in A_{1}$. We shall use the notation $x_{\mathbf{i}_{k}(l)}:=f_{%
\mathbf{i}_{k}(l)}(x)$. Compute%
\begin{equation}
\frac{(2d_{j})^{s}}{\mu _{k}(B(x_{l},d_{j}))}=\frac{(2d_{j})^{s}}{\frac{j+t}{%
m^{k}}}=\frac{m^{k}\left( 2d_{j}\right) ^{s}}{l+t}  \label{fkigual}
\end{equation}%
\textit{only} for those $m^{k-1}(m-1)$ distances $d_{j}$
satisfying the following condition.

\textbf{Condition:} If
\begin{equation*}
d_{j}=dist(x_{l},y)
\end{equation*}%
for some $y\in A_{k}$ such that $y=f_{\mathbf{j}%
_{k}}(z)=f_{j_{1}...j_{k}}(z) $ for some $z\in A_{1}$, then
\begin{equation}
i_{1}(l)\neq j_{1}.  \label{condition}
\end{equation}

\item[3.5] Find the minimum of the $m^{k-1}(m-1)$ values computed in step 3.4.

\item[3.6] Repeat steps 3.1-3.5 for each $x_{l}\in A_{k}$, $l=1,...,m^{k}.$

\item[3.7] Take the minimum of the $m^{k}$ values computed in step 3.6.
\end{itemize}
\end{enumerate}
\end{algorithm}
\begin{algorithm}[General case]
The main difference between this case and the previous one is that the values
of the measures $\mu _{k}$ are different.
The structure 
of the algorithm is the same in both cases.
Consequently, we shall list only the changes needed to find the measure when the contraction ratios are unequal.
\end{algorithm}

\begin{enumerate}
\item[1.] In step 2, replace (\ref{muigulapto}) with
\begin{equation}
\mu _{k}(x)=r_{\mathbf{i}_{k}}^{s}\text{\quad }\forall x\in A_{k},
\label{mukdif}
\end{equation}%
where $\mathbf{i}_{k}=(i_{1},...,i_{k})\in M^{k}$ is the unique sequence of
length $k$ such that $x=f_{\mathbf{i}_{k}}(y)$ for some $y\in A_{1}$.
Thus, $\mu _{k}$ is a probability measure with $spt(\mu
_{k})=A_{k}=\{x_{1},...,x_{m^{k}}\}$. If, for every $j=1,...,m^{k}$, we
denote by $\mathbf{i}_{k}(j)\in M^{k}$ the unique sequence of length $k$
such that $x_{j}=f_{\mathbf{i}_{k}(j)}(y)$ for some $y\in A_{1}$, then we
can write
\begin{equation}
\mu _{k}=\sum_{j=1}^{m^{k}}r_{\mathbf{i}_{k}(j)}^{s}\delta _{x_{j}}\text{.}
\label{sumdeltasdif}
\end{equation}

\item[2.] In step 3.3, replace (\ref{mukigualball}) with%
\begin{equation}\label{mudifball}
\mu _{k}(B(x_{l},d_{j}))=\sum_{q=1}^{j+t}r_{\mathbf{i}_{k}(i_{q})}^{s},
\end{equation}%
where $\mathbf{i}_{k}(i_{q})\in M^{k}$ is such that
$d_{q}=d(x_{l},x_{i_{q}})$ for all $q=1,...,j+t$.

\item[3.] In step 3.4 , replace (\ref{fkigual}) with%
\begin{equation}
\frac{(2d_{j})^{s}}{\mu _{k}(B(x_{l},d_{j}))}=\frac{(2d_{j})^{s}}{%
\sum_{q=1}^{j+t}r_{\mathbf{i}_{k}(i_{q})}^{s}}.  \label{fkdif}
\end{equation}
\end{enumerate}

\begin{rem}
Note that, in (\ref{fkigual}) and (\ref{fkdif}), we only compute the values
of the inverse density function for those distances between points that
belong to different basic cylinder sets. Namely, if $d_{j}=dist(x_{l},x_{p})$
for some $x_{p}\in A_{k}$, then $\frac{\mu _{k}(B(x_{l},d_{j}))}{%
(2d_{l})^{s}}$ is only computed by the algorithm when $x_{l}\in E_{i}$ and $x_{p}\in E_{j}$ with $i\neq j$.
\end{rem}

\begin{notation}
In the rest of the paper we shall use the following notation. Let $%
A_{k}=S\Psi (A_{k-1})$ be the set of $m^{k}$\ points obtained
after $k$ iterations with $A_{1}=\{x\in \mathbb{R}^{n}:f_{i}(x)=x,i=1,...,m\}$. We write $\ $%
\begin{equation*}
A:=\cup _{k=1}^{\infty }A_{k}.
\end{equation*}
For $k\in \mathbb{N}^{+}$ and for each $x_{l}\in A_{k}$, let $D_{k}^{l}$ be the set of $%
m^{k-1}(m-1)$ distances satisfying condition (\ref{condition}) in the
construction of the algorithm, denote by
\begin{equation*}
D_{k}:=\cup _{l=1}^{m^{k}}D_{k}^{j}
\end{equation*}%
the set of the $m^{2k-1}(m-1)$ distances that appear in step 3.5, and write
\begin{equation*}
D:=\cup _{k=0}^{\infty }D_{k}.
\end{equation*}%
Observe that $D$ only takes values in the interval $[c,R]$ (see
(\ref{formulacentred})).
\end{notation}

From now on, we shall assume, without lost of generality, that $R:=|E|=1$.

\section{Convergence of the algorithm}\label{convergsec}

\subsection{Preliminary results.}

The next two lemmas collect some basic results needed in the proofs of our
theorems. We shall prove only those statements that do not follow directly
from the construction of the algorithm.

\begin{lemma}
\label{keylema}
\noindent
\begin{enumerate}
\item[(i)] For every $x\in E$ there exists a sequence $\{x_{k}\} \subset A_{k}$ such that $\lim_{k\rightarrow \infty }x_{k}=x$.


\item[(ii)] For every $k\in \mathbb{N}^{+}$,
\begin{equation*}
A_{k}\subset A_{k+1}.
\end{equation*}

\item[(iii)] Let $k\in \mathbb{N}^{+}$, $x\in A_{k}$, and
$\mathbf{i}_{k}\in M^{k}$ be such that $f_{\mathbf{i}_{k}}(y)=x$
for some $y\in A_{1}$. Then
\begin{equation}
\mu _{k}(x)=\mu _{k}(E_{\mathbf{i}_{k}})=\mu (E_{\mathbf{i}_{k}}).
\label{keyinq}
\end{equation}

\item[(iv)] For every $k\in \mathbb{N}^{+}$,
\begin{equation}
\mu _{k+1}=\mathbf{M}(\mu _{k}).  \label{markrel}
\end{equation}%
Moreover, $\{\mu _{k}\}_{k\in \mathbb{N}^{+}}$ converges weakly to
$\mu $ and thus
\begin{equation}
\lim_{k\rightarrow \infty }\mu _{k}(A)=\mu (A)  \label{convcjtos}
\end{equation}%
for every set $A$ satisfying $\mu (\partial A)=0$.
\end{enumerate}
\end{lemma}

\begin{proof}
(i) and (ii) follow directly from the construction of the algorithm.

\begin{enumerate}
\item[(iii)] By (\ref{mucil}), (\ref{mucilig}), (\ref{muigulapto}),
and (\ref{mukdif}), it suffices to notice that the SSC ensures the
existence, for any cylinder set $E_{\mathbf{i}_{k}}$ of generation
$k$, of a unique point $x\in A_{k}\cap E_{\mathbf{i}_{k}}$. The
existence is clear by construction and the uniqueness holds
because the cardinality of $A_{k}$ is equal to the number of
cylinder sets $E_{\mathbf{i}_{k}}$.

\item[(iv)] Let $x\in A_{k+1}$. By the SSC, there exists a unique $j\in
M $ such that $\ f_{j}^{-1}(x)\in $ $A_{k}$. Moreover, if $\mathbf{i}%
_{k+1}(x)=(i_{1}(x),...,i_{k+1}(x))\in M^{k+1}$ is the unique
sequence of length $k+1$ such that
$x=f_{\mathbf{i}_{k+1}(x)}(y)$ for some $y\in A_{1}$, then
$i_{1}(x)=j$ and thus, by (\ref{markvopdef}) and
(\ref{sumdeltasdif}),
\begin{equation*}
\mathbf{M}(\mu _{k})(x)=r_{j}^{s}\mu _{k}\circ
f_{j}^{-1}(x)=r_{j}^{s}r_{i_{2}(x),...,i_{k+1}(x)}^{s}=r_{\mathbf{i}_{k+1}(x)}^{s}=\mu
_{k+1}(x).
\end{equation*}
This proves (\ref{markrel}).

The weak convergence holds because the sequence $\{\mathbf{M}%
^{k}(\nu )\}$ converges weakly to the invariant measure $\mu $ for every
compactly supported probability measure $\nu $ (see \cite{hutchinson}). In
particular, $\{\mu _{k}\}=$ $\{\mathbf{M}^{k}(\mu _{1})\}$ converges weakly
to $\mu $. Finally, this is equivalent to (\ref{convcjtos}) (see, for
example, \cite{Edgar} Theorem 2.5.11).\qedhere
\end{enumerate}
\end{proof}

\begin{lemma}
\label{attainE}If $(x_{0},d)\in E\times \lbrack c,1]$ is such that%
\begin{equation}
C^{s}E=\frac{(2d)^{s}}{\mu (B(x_{0},d))},  \label{minpair}
\end{equation}%
then $\partial B(x_{0},d)\cap E\neq \emptyset $.
\end{lemma}

\begin{proof}
Suppose that this is not the case. As both $\partial B(x_{0},d)$ and $E$ are
compact sets, there exists $\ 0<\epsilon <dist(E,\partial B(x_{0},d))$
such that
\begin{equation*}
(B(x_{0},d)\setminus B(x_{0},d-\epsilon ))\cap E=\emptyset .
\end{equation*}%
Thus, $\mu (B(x_{0},d))=\mu (B(x_{0},d-\epsilon ))$, contradicting the
minimality of $(x_{0},d)$.
\end{proof}

Observe that for any pair $(x_{0},d)\in E\times \lbrack c,1]$ satisfying (%
\ref{minpair}), Lemma~\ref{attainE} guarantees the existence of a
point $y\in E$ such that
\begin{equation}
C^{s}E=\frac{(2d)^{s}}{\mu (B(x_{0},d))}=\frac{(2dist(x_{0},y))^{s}}{\mu
(B(x_{0},dist(x_{0},y))}.  \label{atta}
\end{equation}

\subsection{Convergence}

In this section we show the convergence of the algorithm described in
section~\ref{description}. We do it in two steps. First, given $x_{0}$, $r$, and $y$
as in (\ref{atta}), we prove, in Theorem \ref{main}, the existence of a
sequence $\{(x_{k},d_{k})\}_{k=1}^{\infty }$ in $A_{k}\times $ $D_{k}$ such
that

\begin{equation*}
m_{k}:=h_{k}(x_{k},d_{k}):=\frac{(2d_{k})^{s}}{\mu _{k}(B(x_{k},d_{k}))}%
\rightarrow \frac{(2d)^{s}}{\mu
(B(x_{0},d))}=h_{s}(x_{0},d)=C^{s}(E).
\end{equation*}%
However, the algorithm's sequence $\{(\tilde{x}_{k},\tilde{d}%
_{k})\}_{k=1}^{\infty }$ is constructed by choosing on the $k$th step a pair $(\tilde{x}%
_{k},\tilde{d}_{k})$ $\in $ $A_{k}\times D_{k}$ such that
\begin{equation}
\tilde{m}_{k}:=h_{k}(\tilde{x}_{k},\tilde{d}_{k}):=\frac{(2\tilde{d}_{k})^{s}%
}{\mu _{k}(B(\tilde{x}_{k},\tilde{d}_{k}))}=\min
\{h_{k}(x_{k},d_{k}):(x_{k},d_{k})\in A_{k}\times D_{k}\}.
\label{seal}
\end{equation}
Actually, for each $k\in \mathbb{N}^{+}$ there might be more
than one pair $(\tilde{x}_{k},\tilde{d}_{k})$ satisfying
(\ref{seal}). Thus, if we chose a sequence $\{(\tilde{x}_{k},
\tilde{d}_{k})\}_{k=1}^{\infty }$ from the set of pairs selected
by the algorithm, it does not need to coincide with the sequence $
\{(x_{k},d_{k})\}_{k=1}^{\infty }$ given in Theorem \ref{main}.
Therefore, one needs to show that the sequence of minimal values
$\tilde{m}_{k}$ converges to the minimum $C^{s}(E)$. In computer experiments we have
observed that in many cases the sequence $h_{k}(\tilde{x}_{k},\tilde{d}%
_{k}) $ is not monotone. Nonetheless, in Theorem \ref{conv} we prove the desired
convergence.

From now on, we use the notation given in \eqref{seal} and write
 $r_{\max }:=\max \{r_{1},...,r_{m}\}$. We refer to the
sequence $\tilde{m}_{k}$ as \emph{the sequence chosen by the
algorithm}.

\begin{thm}
\label{main}Let $x_{0}$, $r$, and $y$ be as in (\ref{atta}). Then, there
exist sequences $\{x_{k}\}$ and $\{y_{k}\}$ such that
$x_{k},y_{k}\in A_{k}$ for all $k \in \mathbb{N}^{+}$, $x_{k}\rightarrow x_{0}$, $y_{k}\rightarrow y$, $%
d_{k}:=dist(x_{k},y_{k})\rightarrow d=dist(x_{0},y)$, and
\begin{equation}
\frac{(2d_{k})^{s}}{\mu _{k}(B(x_{k},d_{k}))}\rightarrow \frac{(2d)^{s}}{\mu
(B(x_{0},d))}.  \label{convergence}
\end{equation}
\end{thm}

\begin{proof}
The existence of the convergent sequences $\{x_{k}\}$, $\{y_{k}\}$, and $%
\{d_{k}\}$ is given by Lemma \ref{keylema} (i).

As $d_{k}\rightarrow d$ and, by condition (\ref{condition}), $\mu
_{k}(B(x_{k},d_{k}))$ is bounded away from zero for all $k \in
\mathbb{N}^{+}$, we need only to show that
\begin{equation}
|\mu _{k}(B(x_{k},d_{k}))-\mu (B(x_{0},d))|\rightarrow 0\text{ as }%
k\rightarrow \infty .  \label{mainclaim}
\end{equation}%
Let $\delta >0$ and let $\phi _{\delta }$ be a compactly supported continuous function on
$\mathbb{R}^{n}$ such that $0\leq \phi
_{\delta }\leq 1$ everywhere, $\phi _{\delta }\equiv 1$ on
$B(x_{0},d+\delta )$, and $\phi _{\delta }\equiv 0$ off of
$B(x_{0},d+2\delta )$. By the weak convergence of $\mu _{k}$ to
$\mu $ (Lemma \ref{keylema} (iv)), we have that
\begin{equation*}
\int \phi _{\delta }d\mu =\lim_{k\rightarrow \infty }\int \phi _{\delta
}d\mu _{k}.
\end{equation*}%
Moreover,  as  $\mu(\partial B(x_0,d))=0$,
\begin{equation*}
|\mu (B(x_{0},d))-\int \phi _{\delta }d\mu |
\end{equation*}%
is small when $\delta $ is small enough. So (\ref{mainclaim}) holds if
\begin{equation*}
|\mu _{k}(B(x_{k},d_{k}))-\int \phi _{\delta }d\mu _{k}|
\end{equation*}%
goes to zero as $\delta \rightarrow 0$ and $k\rightarrow \infty $.
For each $\delta >0$ take $k_{0}=k_{0}(\delta )\in \mathbb{N}^{+}$
such that for all $k\geq k_{0}$, $B(x_{0},d-\delta )\subset
B(x_{k},d_{k})\subset B(x_{0},d+\delta )$, and $r_{\max }^{k}\leq
\delta $.
Hence
\begin{equation}
\int \phi _{\delta }d\mu _{k}\geq \mu _{k}(B(x_{k},d_{k})),  \label{aproxbk}
\end{equation}%
and so
\begin{eqnarray*}
|\mu _{k}(B(x_{k},d_{k}))-\int \phi _{\delta }d\mu _{k}| &=&\int \phi
_{\delta }d\mu _{k}-\mu _{k}(B(x_{k},d_{k}))\leq \\
&\leq &\mu _{k}(B(x_{0},d+2\delta )\setminus B(x_{k},d_{k}))\leq \mu
_{k}(R(2\delta )),
\end{eqnarray*}%
where $R(\delta ):=B(x_{0},d+\delta )\setminus B(x_{0},d-\delta )$.

There remains only to show the convergence of $\mu _{k}(R(2\delta ))$. Let $\{x_{1},...,x_{J}\}$
be the set of points in $A_{k}\cap R(2\delta )$ and,
for $j=1,...,J$, denote by $\mathbf{i}_{k}(j)\in M^{k}$ the unique sequence
of length $k$ satisfying
\begin{equation*}
f_{\mathbf{i}_{k}(j)}(y)=x_{j}\text{ for some }y\in A_{1},
\end{equation*}%
and by $E_{\mathbf{i}_{k}(j)}$ the unique cylinder set of generation $k$
such that $x_{j}\in E_{\mathbf{i}_{k}(j)}$. Then \eqref{sumdeltasdif} and \eqref{keyinq} imply
\begin{eqnarray*}
\mu _{k}(R(2\delta )) &=&\sum_{j=1}^{J}\mu _{k}(x_{j})=\sum_{j=1}^{J}\mu (E_{%
\mathbf{i}_{k}(j)})\leq \\
&\leq &\mu (\{E_{\mathbf{i}_{k}}\subset E\text{ such that }E_{\mathbf{i}%
_{k}}\cap R_{2\delta }\neq \emptyset \text{ and }\mathbf{i}_{k}\in
M^{k}\})\leq \mu (R_{3\delta }).
\end{eqnarray*}%
The last inequality holds because $|E_{\mathbf{i}_{k}}|\leq r_{\max
}^{k}\leq \delta $ for all $\mathbf{i}_{k}\in M^{k}$ with $k\geq k_{0}$.
This completes the proof of (\ref{mainclaim}) as $\mu (\partial
B(x_{0},d))=0 $ (see \cite{mattila0}).
\end{proof}

\begin{thm}
\label{conv}The sequence $\{\tilde{m}_{k}\}_{k\in \mathbb{N}}$
chosen by the algorithm, given in (\ref{seal}), converges to
$C^{s}(E)$. Moreover, for any $k \in \mathbb{N}$ such that
$\mu_k(B(\widetilde{x}_k,\widetilde{d}_k))=\mu(B(\widetilde{x}_k,\widetilde{d}_k))$, there holds
 $$C^s(E)\le \tilde{m}_{k}.$$
\end{thm}

\begin{proof}

The second statement in the theorem is immediate, as
$\mu_k(B(\widetilde{x}_k,\widetilde{d}_k))=\mu(B(\widetilde{x}_k,\widetilde{d}_k))$
and \eqref{censsc} together imply
\begin{equation*}
C^s(E)\le h_{s}(\widetilde{x}_k,\widetilde{d}_k))=
h_{k}(\widetilde{x}_k,\widetilde{d}_k))=\tilde{m}_{k}.
\end{equation*}
Thus, we need only to prove the convergence.

 For $\{\tilde{m}_{k}\}_{k\in \mathbb{N}^{+}}$ as in (\ref{seal}),
let $\{\tilde{m}_{k_{i}}\}_{i\in  \mathbb{N}^{+}}$ be a convergent
subsequence of $\{\tilde{m}_{k}\}_{k\in \mathbb{N}^{+}}$ and
$\tilde{m}:=\lim_{i\rightarrow \infty }\tilde{m}_{k_{i}}$. Since
the algorithm only takes values in the set $A\times (D\cap \lbrack
c,1])$, the sequence $\{\tilde{m}_{k}\}_{k\in \mathbb{N}^{+}}$ is
contained in a compact set and it is enough to show that $\tilde{m}=C^{s}(E)$.

Let $\lambda =\tilde{m}-C^{s}(E)$ and for
$\{\tilde{m}_{k_{i}}\}_{i\in \mathbb{N}^{+}}$, let
$\{m_{k_{i}}\}_{i\in \mathbb{N}^{+}}$ be the corresponding subsequence of the convergent sequence $%
\{m_{k}\}_{k\in
\mathbb{N}
^{+}}$. Then, given $0<\epsilon <|\lambda |/2$, there exists
$i_{0}\in \mathbb{N}^{+}$ such that for any $i\geq i_{0}$,
\begin{eqnarray}
|h_{k_{i}}(x_{k_{i}},d_{k_{i}})-C^{s}(E)| &\leq &\epsilon /2,  \label{tri1} \\
|h_{k_{i}}(\tilde{x}_{k_{i}},\tilde{d}_{k_{i}})-\tilde{m}| &\leq
&\epsilon /2. \label{tri2}
\end{eqnarray}%
First we want to show that $\tilde{m}\leq C^{s}(E)$. If this is not the
case, then $\lambda >\epsilon >0$. Together with (\ref{tri1}) and (\ref%
{tri2}) this implies%
\begin{equation*}
h_{k_{i}}(x_{k_{i}},d_{k_{i}})\leq C^{s}(E)+\epsilon
/2=\tilde{m}-\lambda +\epsilon /2\leq
h_{k_{i}}(\tilde{x}_{k_{i}},\tilde{d}_{k_{i}})-\lambda +\epsilon
<h_{k_{i}}(\tilde{x}_{k_{i}},\tilde{d}_{k_{i}}),
\end{equation*}%
contradicting the minimality of $h_{k_{i}}(\tilde{x}_{k_{i}},\tilde{%
d}_{k_{i}})$. Therefore, $\tilde{m}\leq C^{s}(E).$

With the aim of showing the reverse inequality, suppose that there exists
$i_{1}\in \mathbb{N}^{+}$ such that for any $i\geq i_{1}$,
\begin{equation}
h_{k_{i}}(\tilde{x}_{k_{i}},\tilde{d}_{k_{i}})\geq h_{s}(\tilde{x}_{k_{i}},%
\tilde{d}_{k_{i}})-\epsilon /2,  \label{uniform}
\end{equation}%
and assume that $\tilde{m}>C^{s}(E)$. Then $\lambda <-\epsilon <0$ and thus,
(\ref{tri1}), (\ref{tri2}), and (\ref{uniform}) give
\begin{equation*}
C^{s}(E)=\tilde{m}-\lambda \geq h_{k_{i}}(\tilde{x}_{k_{i}},\tilde{d}%
_{k_{i}})-\epsilon /2-\lambda \geq h_{s}(\tilde{x}_{k_{i}},\tilde{d}%
_{k_{i}})-\epsilon -\lambda
>h_{s}(\tilde{x}_{k_{i}},\tilde{d}_{k_{i}}),
\end{equation*}%
contradicting the minimality of
$h_{s}(x_{0},d)=C^{s}(E)$. There remains only to prove
(\ref{uniform}). Notice that, as $\tilde{d}_{k_{i}}\in \lbrack
c,1]$, (\ref{regularity}) implies that it is enough to show that
for any sequence $\{(x_{k},d_{k})\}_{k=1}^{\infty }$ in $A\times
(D\cap \lbrack c,1])$ and $\delta >0$, there exists
$k_{0}\in \mathbb{N}^{+}$ such that for any $k\geq k_{0}$,
\begin{equation}
|\mu _{k}(B(x_{k},d_{k}))-\mu (B(x_{k},d_{k}))|<\delta .  \label{suniform}
\end{equation}%
For every $k\in \mathbb{N}^{+}$ satisfying that $r_{\max }^{k}<c$,
let
\begin{eqnarray*}
G_{k} &=&\{E_{\mathbf{i}_{k}}:\mathbf{i}_{k}\in M^{k}\text{ and
}E_{\mathbf{i}_{k}}\subset B(x_{k},d_{k})\}, \\
P_{k} &=&\{E_{\mathbf{i}_{k}}:\mathbf{i}_{k}\in M^{k}\text{,
}E_{\mathbf{i}_{k}}\cap B(x_{k},d_{k})\neq \emptyset \text{ and
}E_{\mathbf{i}_{k}}\cap
E\setminus B(x_{k},d_{k})\neq \emptyset \}, \\
R_{k} &=&\{E_{\mathbf{i}_{k}}\in P_{k}:\text{ }A_{k}\cap
E_{\mathbf{i}_{k}}\cap B(x_{k},d_{k})\neq \emptyset \text{ }\}.
\end{eqnarray*}%
Then,
\begin{eqnarray*}
R_{k} &\subset &P_{k}, \\
E\cap B(x_{k},d_{k}) &=&G_{k}\cup (B(x_{k},d_{k})\cap P_{k}), \\
\mu_{k}(B(x_{k},d_{k}))&=&\mu_{k}(G_{k})+\mu_{k}(R_{k}), \qquad
\text{and}
\\
 \mu(B(x_{k},d_{k}))&=&\mu(G_{k})+\mu(B(x_{k},d_{k})\cap P_{k}).
\end{eqnarray*}%
Moreover, by (\ref{keyinq}), $\mu (G_{k})=\mu _{k}(G_{k})$ and $\mu
(R_{k})=\mu _{k}(R_{k})$. All this together with the triangle inequality,
gives
\begin{eqnarray*}
|\mu _{k}(B(x_{k},d_{k}))-\mu (B(x_{k},d_{k}))| &=&|\mu
_{k}(G_{k})+\mu _{k}(R_{k})-\mu (G_{k})-\mu (B(x_{k},d_{k})\cap P_{k})|= \\
&=&|\mu _{k}(R_{k})-\mu (B(x_{k},d_{k})\cap P_{k})|\leq \\
&\leq &\mu (R_{k})-\mu (B(x_{k},d_{k})\cap R_{k})+\mu (\left(
B(x_{k},d_{k})\cap P_{k}\right) \setminus R_{k})\leq \\
&\le & \mu (R_{k})+\mu (\left( B(x_{k},d_{k})\cap P_{k}\right),
\setminus
R_{k})\leq \mu (P_{k})\leq \\
&\leq &\mu (B(x_{k},d_{k}+r_{\max }^{k})\setminus B(x_{k},d_{k}-r_{\max
}^{k})),
\end{eqnarray*}%
where the last inequality holds because \mbox{$P_{k}\subset
B(x_{k},d_{k}+r_{\max}^{k})\setminus B(x_{k},d_{k}-r_{\max }^{k})
$}. Thus, we want to show that $ \mu
(B(x_{k},d_{k}+r_{\max}^{k})\setminus B(x_{k},d_{k}-r_{\max
}^{k}))<\delta $ if $k$ is big enough. This is true, since
$E\times \lbrack c,1]$ is a compact set and the continuity of $\mu
$,  together with the fact that $\mu (\partial B(x,d))=0$ (see
\cite{mattila0}), implies that $\mu(B(x,d+r_{\max }^{k})\setminus B(x,d-r_{\max }^{k}))<\delta $ for any $(x,d)\in E\times \lbrack
c,1]$, if $k$ is big enough. This proves
\eqref{suniform} and concludes the proof of the theorem.
\end{proof}
\begin{rem}\label{upbdd}
Note that the second statement in Theorem \ref{conv} provides
an upper bound for $C^s(E)$. If on the $k$th step the algorithm selects
a ball containing all the $k$th generation cylinder sets that it
intersects, $\widetilde{m}_k$ provides us with an upper bound for
$C^s(E)$.  In general, the existence of such a ball at step $k$
is not guaranteed as it depends on the geometry of the set $E$.
However, $\widetilde{m}_1$ will always provide an upper
bound for $C^s(E)$ as the optimal ball selected on the first
step already contains the whole set. Observe that this value is also an
upper bound for the spherical Hausdorff measure. Even when we
do not have such a ball, we can obtain natural upper bounds
by estimating the value of $\mu
(B(\widetilde{x}_k,\widetilde{d}_k))$, as, for every $k \in
\mathbb{N}$ there holds
\[ C^s(E)\le \frac{(2\widetilde{d}_k)^s}{\mu
(B(\widetilde{x}_k,\widetilde{d}_k))}\]
(see section~\ref{Ejemplos} for further discussion and examples).
\end{rem}
\section{Examples\label{Ejemplos}}

The first part of this section is devoted to showing that the
algorithm recovers easily examples from the literature
in which the centered Hausdorff measure has been computed. We have checked how it
performs for the Cantor type sets analyzed in \cite{daitian},
\cite{zhuzhou}, \cite{zhuzhou0}, and \cite{zhuzhou1}. Here we will
illustrate the accuracy of the algorithm by testing it on three
of these examples. We will also see that the algorithm suggests
the correct minimizing balls for $h_{s}(x,d)$. In the cases where
$C^s(E)$ remains unknown the algorithm still provides candidates
for minimizing balls (and hence for optimal coverings, by
\cite{llorentemoran0} and \cite{llorentemoran1}). These potentially
minimizing balls yields conjectural values for $C^s(E)$ that, in
the cases considered here, are rigorous upper bounds (see Theorem~\ref{conv},
Remark~\ref{upbdd}, Example~\ref{Sr}, and Conjecture~\ref{conj}).
In Remark~\ref{finalrm} we point out that, in some cases, we can
use the balls selected by the algorithm to obtain also the
corresponding lower bounds, although the proofs we have so far are
case specific and are not simple. In the last part of this section we
present other examples not included in the literature together
with some conjectures on the precise values of $C^s (E)$ for these
cases.

Recall that, by Theorem~\ref{tcentered}, our goal is to find the minimal value of $h_{s}(x,r):=\frac{%
(2r)^{s}}{\mu (B(x,r))}$ when $x\in E$ \ and $c\leq r\leq R$. For
this the
algorithm selects, at the $k$th step, some of the pairs $(\tilde{x}_{k},%
\tilde{y}_{k})$ $\in $ $A_{k}\times A_{k}$ satisfying%
\begin{equation*}
\tilde{m}_{k}=h_{k}(\tilde{x}_{k},\tilde{d}_{k})=\frac{(2\tilde{d}_{k})^{s}}{%
\mu _{k}(B(\tilde{x}_{k},\tilde{d}_{k}))}=\min
\{h_{k}(x_{k},d_{k}):(x_{k},d_{k})\in A_{k}\times D_{k}\},
\end{equation*}%
where $\tilde{d}_{k}=|\tilde{x}_{k}-\tilde{y}_{k}|$,
$A_{0}=\{x_{i}\in \mathbb{R}^{n}:f_{i}(x_{i})=x_{i},$
$i=1,...,m\}$ is the set of fixed points for the similitudes in
the IFS, and $A_{k}=S\Psi (A_{k-1})$ (see section~\nolinebreak
\ref{description}). We shall refer to
$B(\tilde{x}_{k},\tilde{d}_{k})$ as the minimizing ball (or
interval).

\begin{enumerate}
\item[1.] $\lambda $\textbf{-Cantor sets in the real line}

Let $K(\lambda )$ be the attractor of the iterated function system $%
\{f_{0}(x)=\lambda x,$ $f_{1}(x)=1-\lambda +\lambda x\}$, $x\in
\lbrack 0,1]$. In \cite{zhuzhou} it is proved that if $0<\lambda \leq \frac{1}{3}$, then%
\begin{equation}
C^{s}(K(\lambda ))=2^{s}(1-\lambda )^{s},  \label{centred1}
\end{equation}%
where $s=\frac{\log 2}{-\log \lambda }$ is the Hausdorff dimension
of $K(\lambda )$.

Let $\lambda = \frac{1}{3}$. Then $K(\frac{1}{3})$ is the
\textbf{middle-third Cantor set}. The following table gives the
results of the algorithm for $K(\frac{1}{3})$ after $14$
iterations.
\begin{equation*}
\begin{tabular}{|c|c|c|c|c|}
\hline
 & $\tilde{m}_{k}$ & $(\tilde{x}_{k},\tilde{y}_{k})$ & $\tilde{d}_{k}$ & $%
B(\tilde{x}_{k},\tilde{d}_{k})$ \\ \hline
\begin{tabular}{l}
\vspace{0.25cm}\\
$ k=0$ \\
\vspace{0.25cm}
\end{tabular}
 & $1.54856$ &
\begin{tabular}{c}
$(0,1)$ \\
\vspace{0.01cm}\\
$(1,0)$%
\end{tabular}
& $1$ & $[-1,1]$ \\ \hline
\begin{tabular}{l}
\vspace{0.25cm}\\
$ k=1$ \\
\vspace{0.25cm}
\end{tabular} & $1.03238$ &
\begin{tabular}{c}
$(\frac{2}{3},\frac{1}{3})$ \\
\vspace{0.01cm}\\
$(\frac{1}{3},\frac{2}{3})$
\end{tabular}
& $0.333333$ &
\begin{tabular}{c}
$\lbrack \frac{1}{3},1]$ \\
\vspace{0.01cm}\\
$\lbrack 0,\frac{2}{3}]$
\end{tabular}
\\ \hline
\begin{tabular}{l}
\vspace{0.25cm}\\
$2 \le k \le 13$ \\
\vspace{0.25cm}
\end{tabular}
 & $1.19902$ &
\begin{tabular}{c}
$(\frac{2}{3},0)$ \\
\vspace{0.01cm}\\
$(\frac{1}{3},1)$
\end{tabular}
& $0.666667$ &
\begin{tabular}{c}
$\lbrack 0,\frac{4}{3}]$ \\
\vspace{0.01cm}\\
$\lbrack -\frac{1}{3},1]$
\end{tabular}
\\ \hline
\end{tabular}
\end{equation*}
Note that, from (\ref{centred1}) with $\lambda =\frac{1}{3}$, it
follows that
\begin{equation*}
C^{s}(K(\frac{1}{3}))=\frac{4^{\frac{\log 2}{\log 3}}}{2}\simeq 1.199023,
\end{equation*}%
so the algorithm has found the exact value of the Hausdorff centered measure
already at the third iteration! A simple computation shows that, for any $%
k\geq 2,$
\begin{equation*}\label{optcantor}
h_{k}(\frac{1}{3},\frac{2}{3})=h_{k}(\frac{2}{3},\frac{2}{3})=\frac{(2\frac{2%
}{3})^{s}}{1}=\frac{4^{s}}{2}=C^{s}(K(\frac{1}{3})).
\end{equation*}
Notice also that $h_{k}(\frac{1}{3},\frac{2}{3})$ does not depend on
$k$. In fact, one checks easily that
\begin{equation*}
h_{s}(\frac{1}{3},\frac{2}{3})=h_{s}(\frac{2}{3},\frac{2}{3})=C^{s}(K(\frac{1}{3})).
\end{equation*}

Thus, $B(\frac{1}{3},\frac{2}{3})=[-\frac{1 }{3},1]$ and
$B(\frac{2}{3},\frac{2}{3})=[0,\frac{4}{3}]$ are the actual
minimizing intervals for $h_{s}(x,r)$ and the algorithm has found them
already on the second step.

\item[2.] $(\lambda _{1},\lambda _{2})$\textbf{-symmetry Cantor
sets in the real line}

Let $K(\lambda _{1},\lambda _{2})$ be the symmetric Cantor set defined as the
attractor of the iterated function system $\{f_{0}(x)=\lambda _{1}x,$ $%
f_{1}(x)=\lambda _{2}x+\frac{1-\lambda _{2}}{2},$ $f_{2}(x)=1-\lambda
_{1}+\lambda _{1}x\}$, $x\in \lbrack 0,1]$. By \cite{daitian} we know that
if
\begin{equation}
\frac{1-2\lambda _{1}-\lambda _{2}}{2}\geq \lambda,  \label{conditiondai}
\end{equation}%
then
\begin{equation}
C^{s}(K(\lambda _{1},\lambda _{2}))=1,  \label{centred2}
\end{equation}
where $\lambda \equiv max\{\lambda _{1},\lambda _{2}\}$.

Here is the output of the algorithm when $\lambda
_{1}=\frac{1}{8}$ and $ \lambda _{2}=\frac{1}{5}$:
\begin{equation*}
\begin{tabular}{|c|c|c|c|c|}
\hline
 & $\tilde{m}_{k}$ & $(\tilde{x}_{k},\tilde{y}_{k})$ & $\tilde{d}_{k}$ & $%
B(\tilde{x}_{k},\tilde{d}_{k})$ \\ \hline
\begin{tabular}{l}
\vspace{0.1cm}\\
$ k=0$ \\
\vspace{0.1cm}
\end{tabular} & $1$ & $(\frac{1}{2},1)$ & $0.5$ & $[0,1]$ \\ \hline
\begin{tabular}{l}
\vspace{0.1cm}\\
$ k=1$ \\
\vspace{0.1cm}
\end{tabular} & $1$ & $(\frac{1}{2},1)$ & $0.5$ & $[0,1]$ \\ \hline
\begin{tabular}{l}
\vspace{0.1cm}\\
$ k=2$ \\
\vspace{0.1cm}
\end{tabular} &
$1$ & $(\frac{1}{2},1)$ & $0.5$ & $[0,1]$ \\ \hline
\begin{tabular}{l}
\vspace{0.1cm}\\
$ k=3$ \\
\vspace{0.1cm}
\end{tabular} & $1$ &
$(\frac{1}{2},1)$ & $0.5$ & $[0,1]$ \\ \hline
\end{tabular}%
\end{equation*}%
There is no need to iterate more as this stability is empirical
evidence that the algorithm has found the optimal ball. Indeed, we
have tested the algorithm with this kind of symmetric Cantor sets
for different values of $\lambda_1$ and $\lambda_2$ and, in all
our trials, $[0,1]$ appears as the minimizing interval for
every $k \ge 0$. Thus, $[ 0,1]$ is the natural candidate to be the
minimizing ball for $h_{s}$ as well. In fact, it is easy to use the exact
value of  $C^{s}(K(\lambda _{1},\lambda _{2}))$ given in
\cite{daitian} (see (\ref{centred2})) to check that this is
the case. To see this, notice that for any $\lambda _{1},\lambda
_{2}$ as in (\ref{conditiondai}), $\ A_{0}=\{0,\frac{1}{2},1\}$
and
\begin{equation*}
h_{k}(\frac{1}{2},\frac{1}{2})=h_{s}(\frac{1}{2},\frac{1}{2})=\frac{(2\frac{1}{2})^{s}}{1}=1=C^{s}(K(\lambda
_{1},\lambda _{2})).
\end{equation*}

 \item[3.] \textbf{Cantor type sets in the plane}

Let $S$ be the attractor of the iterated function system $%
\{f_{1},f_{2},f_{3},f_{4}\}$ where $f_{i}(x)=\lambda _{i}x+b_{i}$, $%
i=1,2,3,4 $, $x=(x_{1},x_{2})\in \mathbb{R}^{2}$, $0<\lambda
_{i}\leq \frac{1}{2+\sqrt{2}}$, $b_{1}=(0,0)$, $b_{2}=(1-\lambda
_{2},0)$, $b_{3}=(1-\lambda _{3},1-\lambda _{3})$, and
$b_{4}=(0,1-\lambda _{4})$. In \cite{zhuzhou1} it is proved that if the parameters $\lambda _{1}$,
$\lambda _{2}$, $\lambda _{3}$, $\lambda _{4}$, satisfy the
conditions

\begin{enumerate}
\item[(i)] $2\lambda _{p}\leq (\frac{(1-\lambda _{k})d_{\min
}^{p}}{\sqrt{2^{s}}})^{\frac{1}{1-s}}$, where $k\neq p$, $k$,
$p\in \{1,2,3,4\},$

\item[(ii)] $\frac{\lambda _{k}^{s}+\lambda _{i}^{s}+\lambda
_{j}^{s}}{ (1-\lambda _{k})^{s}}\leq \frac{1}{\sqrt{2^{s}}}$,
where $k\neq i$, $k\neq j$ , $i\neq j$, $k$, $i$, $j\in
\{1,2,3,4\}$,
\end{enumerate}
then
\begin{equation}\label{centreds}
C^{s}(S)=D_{\max }^{-1},
\end{equation}
where%
\begin{eqnarray}
D_{\max } &=&\max_{1\leq t\leq 4}\left\{ \frac{1}{(2\sqrt{2}(1-\lambda
_{t}))^{s}}\right\}  \label{centred3} \\
d_{\min }^{k} &=&\min \left\{ \frac{\sqrt{2^{s}}\lambda _{k}^{s}}{(\min
\{1-\lambda _{i},1-\lambda _{j}\})^{s}},\frac{\lambda _{k}^{s}+\lambda
_{i}^{s}+\lambda _{j}^{s}}{(\sqrt{2}(1-\lambda _{p}))^{s}}\right\}  \notag \\
(k,i,j,p) &\in &\{(1,2,3,4),(2,1,3,4),(3,2,4,1),(4,1,3,2)\}  \notag
\end{eqnarray}%
and
\begin{equation*}
0<\dim _{H}(S)=s<1.
\end{equation*}
Next, we test the algorithm for the case $\lambda _{1}=\lambda _{3}=%
\frac{1}{400}$ and $\lambda _{2}=\lambda _{4}=\frac{1}{20}$. In
this case we have that \nolinebreak{$s=\log
_{20}(\sqrt{3}+1)\approx 0.335495$} and
\begin{equation*}
C^{s}(S)=\left( \frac{19\sqrt{2}}{10}\right) ^{s}\approx 1.\,393\,2
\end{equation*}
see \cite{zhuzhou1}.

The next table shows the results given by the algorithm after six
iterations.

\begin{table}[!ht]
  \centering
  \caption{}
  Output of the algorithm for the case $\lambda _{1}=~\lambda _{3}=
~\frac{1}{400}$ and $\lambda _{2}=\lambda _{4}=\frac{1}{20}$.
\label{tabla3}

\vspace{0.1 cm}

\begin{tabular}{|c|c|c|c|c|}
\hline $k$ & $\tilde{m}_{k}$ & $(\tilde{x}_{k},\tilde{y}_{k})$ &
$\tilde{d}_{k}$ & $ B(\tilde{x}_{k},\tilde{d}_{k})$ \\ \hline
\begin{tabular}{l}
\vspace{0.25cm}\\
$ k=0$ \\
\vspace{0.25cm}
\end{tabular}
& $1.4174$ &
\begin{tabular}{l}
$((0,0),(1,1))$ \\
\vspace{0.01cm}\\
 $((1,0),(0,1))$
\end{tabular}
& $\sqrt{2}=1.41421$ &
\begin{tabular}{l}
$B((0,0),\sqrt{2})$ \\
\vspace{0.01cm}\\
$((1,0),(0,1))$
\end{tabular}
\\ \hline
\begin{tabular}{l}
\vspace{0.25cm}\\
$1 \le k \le 5$ \\
\vspace{0.25cm}
\end{tabular}
 & $1.39321$ &
\begin{tabular}{l}
$((\frac{1}{20},\frac{19}{20}),(1,0))$ \\
\vspace{0.01cm}\\
$((\frac{19}{20},\frac{1}{20}),(0,1))$%
\end{tabular}
& $\frac{19\sqrt{2}}{20}=1.3435$ &
\begin{tabular}{l}
$B((\frac{1}{20},\frac{19}{20}),\frac{19\sqrt{2}}{20})$ \\
\vspace{0.01cm}\\
$B((\frac{19}{20},\frac{1}{20}),\frac{19\sqrt{2}}{20})$
\end{tabular}
\\ \hline
\end{tabular}
\end{table}

We note that the pattern of the previous examples is repeated
here. There exists $k_0$ such that, for all $k \ge k_0$, the same
balls are selected. The results in Table \ref{tabla3} show that
for all $k \ge 1$, the
optimal balls for the discrete density function $h_{k}$ are $B((\frac{1}{20},\frac{19}{20}),\frac{19\sqrt{2}}{20})$ and
$B((\frac{19}{20},\frac{1}{20}),\frac{19 \sqrt{2}}{20})$, so these
should be considered as the natural candidates to minimize $h_{s}$.
As before, by (\ref{centreds}), a simple computation shows that
these candidates are in fact the minimizing balls:
\begin{eqnarray}\label{minejem}
\nonumber
h_{s}((\frac{1}{20},\frac{19}{20}),\frac{19\sqrt{2}}{20})&=&
h_{s}((\frac{19}{20},\frac{1}{20}),\frac{19\sqrt{2}}{20})= \\
&=&h_{k}((\frac{1}{20},\frac{19}{20}),\frac{19\sqrt{2}}{20})=
h_{k}((\frac{19}{20},\frac{1}{20}),\frac{19\sqrt{2}}{20})
 =\frac{\left( 2\frac{19\sqrt{2}}{20}%
\right) ^{s}}{1}= \\ \nonumber&=&\left(
\frac{19\sqrt{2}}{10}\right) ^{s}=C^{s}(S).
\end{eqnarray}
The third equality in (\ref{minejem}) holds because $\mu
_{k}(B((\frac{1}{20},\frac{19}{20}),\frac{19\sqrt{2}}{20}))=\mu
_{k}(B((\frac{19}{20},\frac{1}{20}),\frac{19\sqrt{2}}{20}))=1$ for any $k \geq 1$.

\end{enumerate}

\begin{rem}
\label{finalrm} As we have seen in the preceeding examples the
algorithm is also helpful for finding natural candidates for optimal
balls. For instance, an output repeated for all iterations $k\geq
k_{0}$ is a natural candidate. Once the algorithm has suggested
such an optimal ball, this can be used to provide new proofs of (\ref{centred1}), (%
\ref{centred2}), and (\ref{centred3}). We explain the main idea for
the first case. The other cases can be treated in a similar fashion.

$\lambda $\textbf{-Cantor sets in the real line}.
We noticed that when we applied the algorithm to different $K(\lambda)$ we obtained two values for
$(\tilde{x}_{k},\tilde{y}_{k})$, namely $(\lambda ,1-\lambda )$
and $(1-\lambda ,1-\lambda )$. Once we have a candidate for the
optimal ball,\ we can recover (\ref{centred1}) by showing that\
the minimum of $h_{s}(x,r)$ is attained for the pairs $(\lambda
,1-\lambda )$
and $(1-\lambda ,1-\lambda )$, i.e.%
\begin{eqnarray*}
C^{s}E &=&\min \left\{ \frac{(2r)^{s}}{\mu (B(x,r))}:x\in E\text{ \ and }%
c\leq r\leq R\right\} =\frac{(2(1-\lambda ))^{s}}{\mu (B(\lambda ,1-\lambda
))} \\
&=&\frac{(2(1-\lambda ))^{s}}{\mu (B(1-\lambda ,1-\lambda ))}
=2^{s}(1-\lambda )^{s}.
\end{eqnarray*}
The proof reduces to showing that
\begin{equation}
\frac{(2r)^{s}}{\mu (B(\lambda ,r))}\geq 2^{s}(1-\lambda )^{s}\text{\qquad
for any }r\in \lbrack 1-2\lambda ,1-\lambda ],  \label{cantors}
\end{equation}%
since the other cases are trivial. The inequality \eqref{cantors} can be
proved by hand using the geometry of the set.

For the $(\lambda _{1},\lambda _{2})$-symmetry Cantor sets in the
real line the same method can be used to show that the optimal
interval is actually the one chosen by the algorithm, namely the
interval $[0,1]$. Finally, for the attractor $S$ of the third
example, the candidate for an optimal ball is $B((\lambda
_{1},\lambda _{1}),\sqrt{2}(1-\lambda _{1}))$ whenever $\lambda
_{1}\geq \lambda _{2}\geq \lambda _{3}\geq \lambda _{4}$.

We note that the results in \cite{daitian}, \cite{zhuzhou},
\cite{zhuzhou0}, and \cite {zhuzhou1} are obtained using the
relation between the upper spherical density of the invariant
measure and the centered Hausdorff measure. More precisely, in these references it is shown, that from
the fact that, if $C^{s}(E)<\infty $, there holds
\begin{equation*}
\bar{d}^{s}(C_{|E}^{s},x)=\lim \sup_{r\rightarrow
0}\frac{C^{s}(B(x,r)\cap E)}{(2r)^{s}}=1,
\end{equation*}
for almost all $x\in E$ (see \cite{tricot}), there follows
\begin{equation*}
C^{s}(E)=(\bar{d}^{s}(\mu ,x))^{-1},
\end{equation*}
for almost all $x \in E$, where $\mu =\frac{C_{|E}^{s}}{C^{s}(E)}$ is the invariant measure.
So they proved that
\begin{equation*}
\bar{d}^{s}(\mu ,x)=\lim \sup_{r\rightarrow 0}\frac{1}{h_{s}(x,r)}=\frac{1}{%
2^{s}(1-\lambda )^{s}},
\end{equation*}
for every point $x$ in a certain set of full
$\mu $ measure.
Notice that the method suggested by the algorithm is different since it is
not necessary to pass to the limit.
\end{rem}
Next, we present  the results obtained by the algorithm for a
family of Sierpinski gaskets whose centered Hausdorff measure is
not known.
\begin{exmp}[Sierpinski gasket $S(r)$]\label{Sr}
Let $S(r)$ be the attractor of the system $\Psi=\{f_0, f_1,f_2\}$,
where
\begin{eqnarray}\label{siersim}
f_0(x,y)&=&r(x,y),\\
f_1(x,y)&=&r(x,y)+(1-r,0), \quad \textrm{ and} \\
f_2(x,y)&=&r(x,y)+(\frac{1}{2}(1-r), (1-r)\frac{\sqrt{3}}{2})).
\end{eqnarray}
After testing many examples, we have discovered that for $r<0.25$
the algorithm quickly stabilizes at $
[2(1-r)(r^2+r+1)^{\frac{1}{2}}]^s $, where $s=\frac{-\log 3}{\log
r}$ is the Hausdorff dimension of $S(r)$. We conjecture
that, in these cases,
$$ C^{s}(S(r))=[2(1-r)(r^2+r+1)^{\frac{1}{2}}]^s. $$
We illustrate this with Table \ref{table2}, which shows the output of the algorithm
 for $r=0.2$.
\begin{table}[H]\label{table2}
  \centering
  \caption{}
  Output of the algorithm for the case $r=0.2$.
\label{tabla4}

 \vspace{0.1 cm}

\begin{tabular}{|c|c|c|c|c|}
\hline $k$ & $\tilde{m}_{k}$ & $(\tilde{x}_{k},\tilde{y}_{k})$ &
$\tilde{d}_{k}$ & $ B(\tilde{x}_{k},\tilde{d}_{k})$ \\
\hline
\begin{tabular}{l}
\vspace{0.25cm}\\
$ k=0$ \\
\vspace{0.25cm}
\end{tabular}
& $1.60504$ &
\begin{tabular}{l}
\vspace{0.2cm}
$((\frac{5}{10},\frac{\sqrt{3}}{2}),(1,0))$ \\
$((\frac{5}{10},\frac{\sqrt{3}}{2}),(0,0))$
 \vspace{0.2cm}
\end{tabular}
& $1$ &
\begin{tabular}{l}
\vspace{0.2cm}
$B((\frac{5}{10},\frac{\sqrt{3}}{2}),1)$ \\
$B((\frac{5}{10},\frac{\sqrt{3}}{2}),1)$
\vspace{0.2cm}
\end{tabular}
\\ \hline
\begin{tabular}{l}
\vspace{0.25cm}\\
$ k=1$ \\
\vspace{0.25cm}
\end{tabular}
& $1.51231$ &
\begin{tabular}{l}
\vspace{0.2cm}
$((\frac{2}{10},0),(\frac{5}{10},\frac{\sqrt{3}}{2}))$ \\
$((\frac{8}{10},0),(\frac{5}{10},\frac{\sqrt{3}}{2}))$
\vspace{0.2cm}
\end{tabular}

 & $0.916515$ &
\begin{tabular}{l}
\vspace{0.2 cm}
$B((\frac{2}{10},0), \frac{\sqrt{21}}{5})$ \\
$B((\frac{8}{10},0),\frac{\sqrt{21}}{5})$
\end{tabular}\\
 \hline
\begin{tabular}{l}
 \vspace{0.05cm}\\
$2 \le k \le 8$ \\
\vspace{0.05cm}
\end{tabular}
& $1.48326$ &

\begin{tabular}{l}
$((\frac{9}{50},\frac{\sqrt{3}}{50}),(\frac{5}{10},\frac{\sqrt{3}}{2}))$
\end{tabular}
 & $0.890842$ &
\begin{tabular}{l}
$B((\frac{9}{50},\frac{\sqrt{3}}{50}),\frac{4\sqrt{31}}{25})$
\end{tabular}
\\ \hline
\end{tabular}
\end{table}
The pattern observed in our trials for these Sierpinski gaskets is
the following. For any $k \ge 2$ the optimal ball chosen by the
algorithm is
$B(\bar{x},d(\bar{x},\bar{x_2}))=B(\bar{x},(1-r)\sqrt{r^2+r+1}),$
where $\bar{x}=f_0(f_1(\bar{x}_2))$ and
$\bar{x}_2=f(\bar{x}_2)=f(x_2,y_2) \in \mathbb{R}^2$ is the fixed
point of $f_2(x,y)$ (see Figure~\ref{sier}). Observe that  $S(r)
\subset B(\bar{x},d(\bar{x},\bar{x_2}))$. Hence, for any $k \in
\mathbb{N}$,
$$ \mu_k(B(\bar{x},d(\bar{x},\bar{x_2})))=
\mu(B(\bar{x},d(\bar{x},\bar{x_2})))=1.$$
Then, by
Theorem~\ref{upbdd}, we have a rigorous upper bound for
$C^s(S(r))$, namely
\[C^s (S(r))\le [2(1-r)(r^2+r+1)^{\frac{1}{2}}]^s \qquad \forall \ r<0.25 \ .\]

\begin{figure}[h]
  \includegraphics[width=1.5\textwidth]{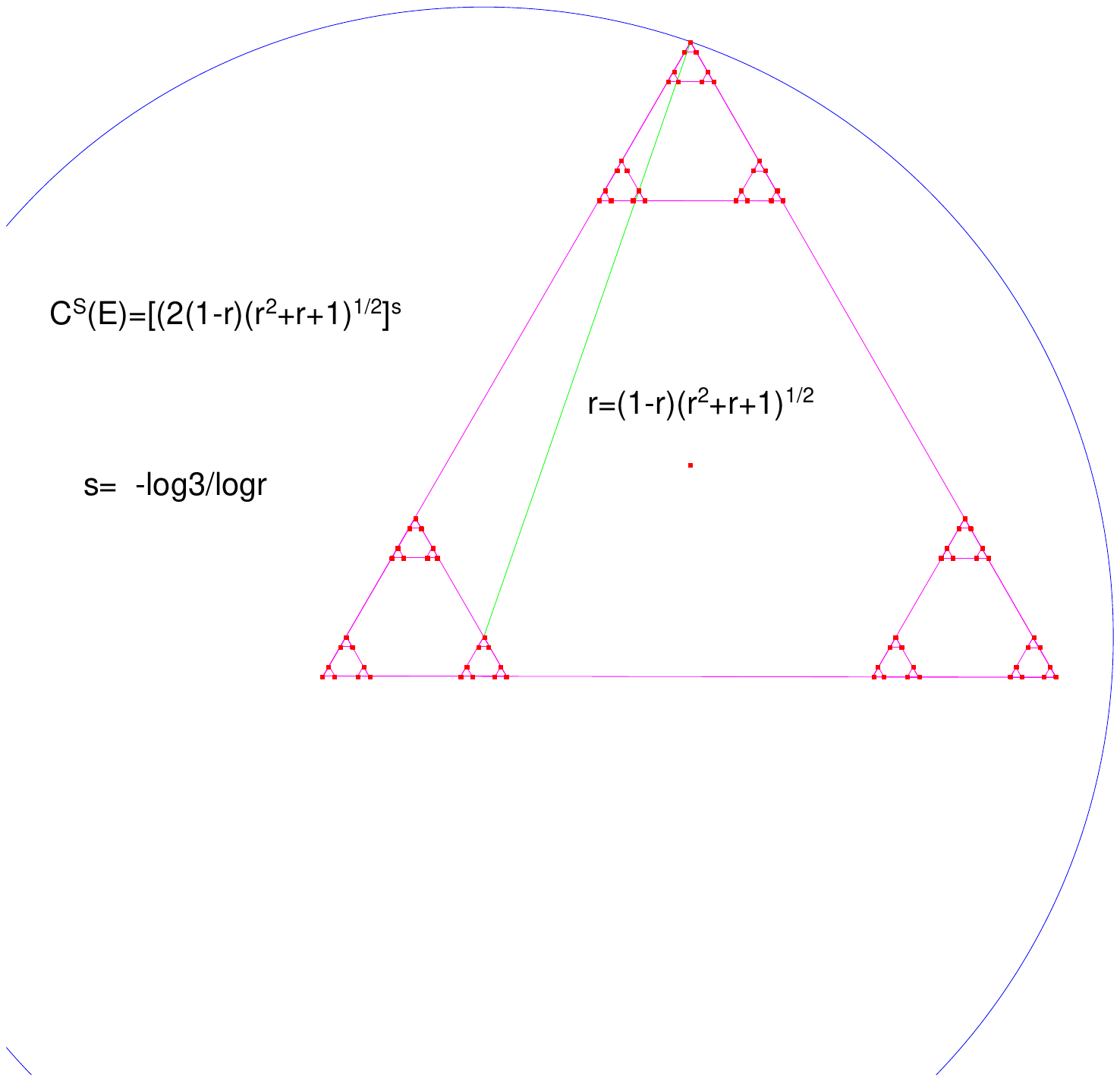}
  \caption{}[Minimizing ball selected by the algorithm for $S(r)$
with $r<~0.25$.]\label{sier}
\end{figure}

\begin{rem}
Note that in all previous examples the minimizing ball (or
interval) is big enough to cover the whole set. That is, if $B$ is
the optimal ball, then $\mu(B)=1$. The reason for this lies in the
fact that the contraction factors are very small compared with the
distances that separate the cylinder sets.

In the case of the Sierpinski gaskets $S(r)$ (Example~\ref{Sr}),
if $r<0.25$, then the candidate for an optimal ball is selected
already at the second step because $r$ is still quite small
in comparison with $c$ (see Theorem~\ref{tcentered} for notation). With larger values for $r$, the centers of the
selected balls are the same after the second step, but the end
point varies slightly. Therefore, the value of $\widetilde{m}_k$
changes slightly after the second step, but in a few steps
the output is stable to three decimal places. This leads, for
example, to the following conjectures for the Sierpinski gasket
$S(\frac{1}{3})$ and the well-known planar $\frac{1}{4}$-Cantor set, $C(\frac{1}{4})$, both self-similar sets with dimension
$1$.
\end{rem}
\newpage
\begin{conj}\label{conj}
Let $S(\frac{1}{3})$ be the attractor of the system
$\Psi=\{f_0,f_1,f_2\}$ given in (\ref{siersim}) with
$r=\frac{1}{3}$ and let $C(\frac{1}{4})$ be the attractor of the system
$\tilde{\Psi}=\{g_0,g_1,g_2, g_3\}$ where
\begin{eqnarray*}
g_0(x,y)&=&\frac{1}{4}(x,y),\\
g_1(x,y)&=&\frac{1}{4}(x,y)+(\frac{3}{4},0), \\
g_2(x,y)&=&\frac{1}{4}(x,y)+(0,\frac{3}{4}),\\
g_3(x,y)&=&\frac{1}{4}(x,y)+(\frac{3}{4},\frac{3}{4}).
\end{eqnarray*}
The conjectural values given by the algorithm for these
self-similar sets are
\begin{enumerate}
    \item $C^1 (S(1/3))\simeq 1.537$.
    \item $C^1(C(\frac{1}{4}))\simeq 1.95$.
\end{enumerate}
We will analyze the planar  $\frac{1}{4}$-Cantor set to show
how to obtain a rigorous upper bound from the results given by
the algorithm. The Sierpinski gasket $S(\frac{1}{3})$ can be treated
in a similar manner.

First we present a table showing the results of seven iterations of the algorithm (see Table~\ref{tabla5}).
For simplicity we
focus on only one of the minimizing balls obtained. In this case
we believe that the description of the pairs
$(\tilde{x}_{k},\tilde{y}_{k})$ by their corresponding codes is
more illustrative than the description by their coordinates, so we will use the
following standard notation: for every $k \in \mathbb{N}$ and
$i_1i_2...i_k \in M^k$,
$$x_{i_1i_2...i_k}=g_{{i_1i_2...i_{k-1}}}(x_{i_k})=g_{i_1}\circ g_{i_2}\circ ...\circ g_{i_{k-1}}(x_{i_k}),$$
where $x_{i_k}=g_{i_k}(x_{i_k})$, $M=\{0,1,2,3\}$, and
$M^k=\{i_1i_2...i_k : i_j \in M \ \forall \  j=1,...,k\}$. Observe
that if $i_t = i_{t+1}=...=i_k$ for some $0<t\le k$, then
$x_{i_1i_2...i_k}=x_{i_1i_2...i_t}$. We shall always use the
shorter notation.

\begin{table}[H]
  \centering \caption{} Output of the algorithm for
  $C(\frac{1}{4})$. \label{tabla5}

 \vspace{0.1 cm}

\begin{tabular}{|c|c|c|c|}
\hline $k$ & $\tilde{m}_{k}$ & $(\tilde{x}_{k},\tilde{y}_{k})$ &
$\tilde{d}_{k}$ \\
\hline
\begin{tabular}{l}
\vspace{0.05cm}\\
 $ k=0$ \\
 \vspace{0.05cm}
\end{tabular}
& $2.66667$ &
\begin{tabular}{l}
\vspace{0.05cm}\\
 $(x_0,x_3)$\\
\vspace{0.05cm}
\end{tabular}
& $1$
\\ \hline
\begin{tabular}{l}
\vspace{0.05cm}\\
 $ k=1$\\
 \vspace{0.05cm}
\end{tabular}
& $1.92296$ &
\begin{tabular}{l}
\vspace{0.05cm}\\
 $(x_{20}, x_{03})$\\
 \vspace{0.05cm}
\end{tabular}

 & $0.901388$

\\
 \hline
\begin{tabular}{l}
 \vspace{0.05cm}\\
$ k =2$\\
 \vspace{0.05cm}
\end{tabular}
& $1.95814$ &

\begin{tabular}{l}
\vspace{0.05cm}\\
$(x_{20}, x_{020})$\\
\vspace{0.05cm}
\end{tabular}
 & $0.795495$
\\ \hline
\begin{tabular}{l}
 \vspace{0.05cm}
$ k =3$
\vspace{0.05cm}
\end{tabular}
& $1.95542$ &

\begin{tabular}{l}
\vspace{0.05cm}\\
$(x_{20}, x_{3})$\\
\vspace{0.05cm}
\end{tabular}
 & $0.790569$

\\ \hline
\begin{tabular}{l}
 \vspace{0.05cm}\\
$ k =4$ \\
\vspace{0.05cm}
\end{tabular}
& $1.95306$ &

\begin{tabular}{l}
\vspace{0.05cm}\\
$(x_{20}, x_{3})$\\
\vspace{0.05cm}
\end{tabular}
 & $0.790569$

\\ \hline
\begin{tabular}{l}
 \vspace{0.05cm}\\
$ k =5$ \\
\vspace{0.05cm}
\end{tabular}
& $1.95388$ &

\begin{tabular}{l}
\vspace{0.05cm}\\
$(x_{20}, x_{020023})$\\
\vspace{0.05cm}
\end{tabular}
 & $0.790662$

\\ \hline
\begin{tabular}{l}
 \vspace{0.05cm}\\
$ k =6$ \\
\vspace{0.05cm}
\end{tabular}
& $1.95417$ &

\begin{tabular}{l}
\vspace{0.05cm}\\
$(x_{20}, x_{020021})$\\
\vspace{0.05cm}
\end{tabular}
 & $0.790662$
\\ \hline
\end{tabular}
\end{table}

For simplicity, we are going to use the value of $\widetilde{m}_3$
in Theorem~\ref{conv} to obtain the upper bound for
$C^1(C(\frac{1}{4}))$. It easily seen that
$\mu_3(B(\widetilde{x}_3,\widetilde{d}_3))=\mu(B(\widetilde{x}_3,\widetilde{d}_3))=\frac{3}{4}+\frac{3}{4^3}+\frac{3}{4^4}$
(see Figure~\ref{cant}), so by  Theorem~\ref{conv} we have that
\begin{equation*}
C^1(C(\frac{1}{4}))\le
f_3(\widetilde{x}_3,\widetilde{d}_3)=1.95542.
\end{equation*}

\begin{figure}[h]
  \includegraphics[width=0.7\textwidth]{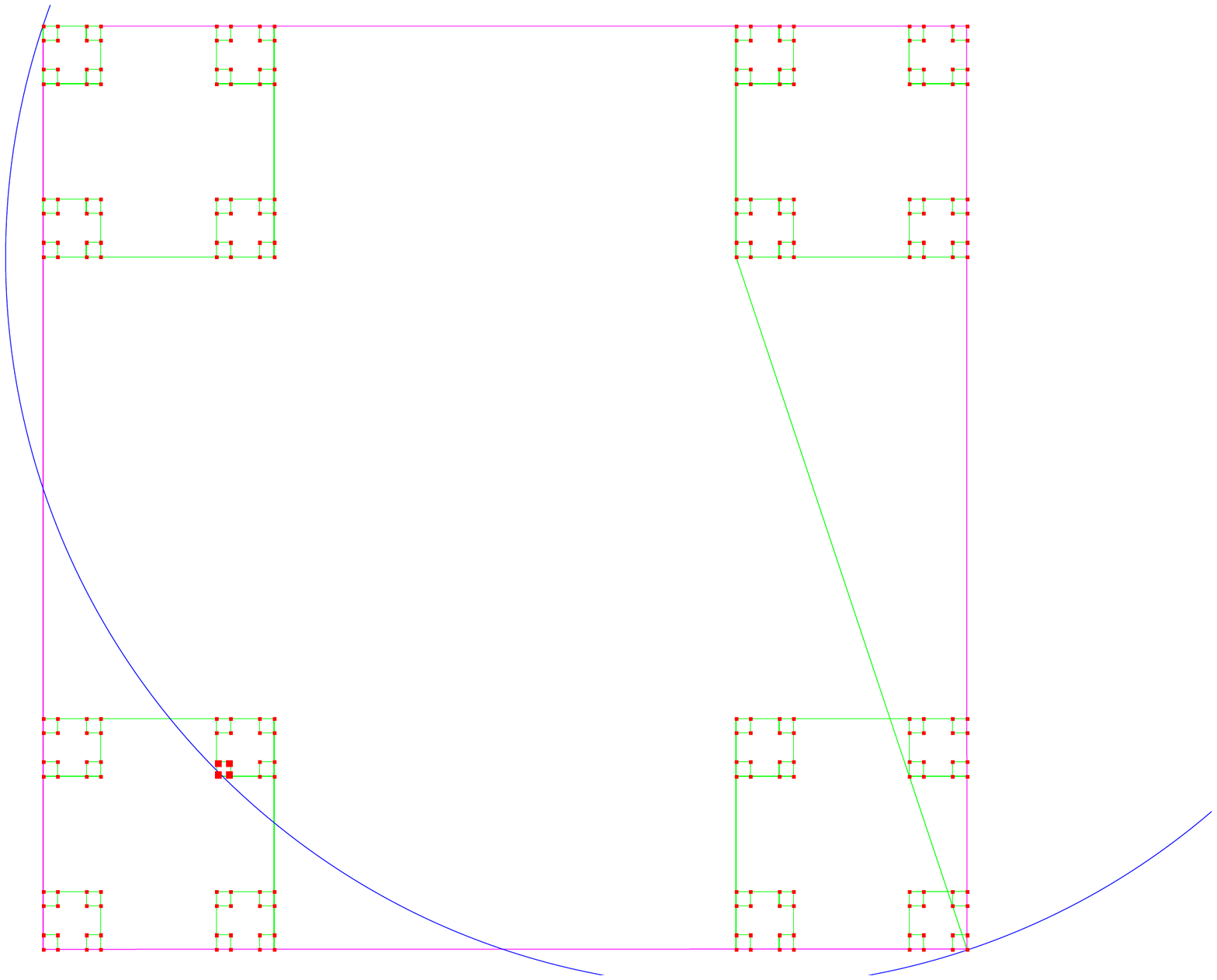}
  \caption{}[$B(\widetilde{x}_4,\widetilde{d}_4)$: Minimizing ball selected by the algorithm  at step $k=4$ for $C(\frac{1}{4})$.]\label{cant}
\end{figure}

\end{conj}
\end{exmp}




\newpage
\begin{thebibliography}{99}
\bibitem{ayer} Ayer, E. and R. Stricharz, (1999), \textit{Exact Hausdorff
measure and intervals of maximum density for Cantor sets, }Transactions of
the American Mathematical Society 351, num. 9, 3725-3741.

\bibitem{Barnsley} Barnsley, M. Academic Press Inc.,
Boston, MA, 1988. xii+396 pp.

\bibitem{hutchinson} Hutchinson, J. E. (1981), \textit{Fractals and
self-similarity, }Indiana Univ. Math. J. 30, 713-747.

\bibitem{daitian} Dai, M. and Tian, L. (2005), \textit{Exact
Hausdorff centered measure of symmetry Cantor sets}. Chaos
Solitons Fractals 26 , no. 2, 313--323.

\bibitem{Edgar} Edgar, G. A. (1998), \textit{Integral, probability, and
fractal measures.} Springer-Verlag, New York, x+286 pp.

\bibitem{falconer} Falconer, K. J. (1986) \textit{The geometry of fractal
sets.} Cambridge Tracts in Mathematics, 85. Cambridge University Press,
Cambridge,. xiv+162 pp.

\bibitem{jiazhouzhu} Jia, B.; Zhou, Z. and Zhu, Z. (2002) \textit{A lower
bound for the Hausdorff measure of the Sierpinski gasket.} Nonlinearity 15 ,
no. 2, 393--404.

\bibitem{jiazhouzhuluo} Jia, B.; Zhou, Z.; Zhu, Z. and Luo, J. (2003)
\textit{The packing measure of the Cartesian product of the middle third
Cantor set with itself}. J. Math. Anal. Appl. 288, no. 2, 424--441.

\bibitem{llorentemoran0} Llorente, M. and Mor\'{a}n, M. (2007), \textit{%
Self-similar sets with optimal coverings and packings}, J. Math. Anal. Appl.
334, 1088-1095.

\bibitem{llorentemoran1} Llorente, M. and Moran, M. (2010), \textit{"Advantages of
the centered Hausdorff measure from the computability point of
view"} Math. Scand. 107, no. 1, 103-122.

\bibitem{mattila0} Mattila, P. (1982), \textit{On the structure of
self-similar fractals,} Annales Academiae \ Scientiarum Fennicae,
Series A. I. Mathematica, num. 7, 189-195.

\bibitem{meifeng&} Meifeng, D. and Lixin,T. (2005) \textit{Exact Hausdorff
centered measure of symmetry Cantor sets}. Chaos Solitons Fractals 26, no.
2, 313--323

\bibitem{moran1} Mor\'{a}n, M. (2005) \textit{Computability of the Hausdorff
and packing measures on self-similar sets and the tiling self-similar
principle}, Nonlinearity 18 559-5709

\bibitem{sainttricot} Saint Raymond, X. ; Tricot, C. (1988) \textit{Packing
regularity of sets in }$\mathit{n}$\textit{-space.} Math. Proc.
Cambridge Philos. Soc. 103, no. 1, 133-145

\bibitem{tricot0} Tricot, C.(1982) \textit{Two definitions of fractional
dimension}, Math. Proc. Cambridge Philos. Soc. 91, 57-74.

\bibitem{tricot} Tricot, C. (1991) \textit{Rectifiable and Fractal Sets}, in
Fractal Geometry and Analysis, Kluwer Academic Publishers, 367-403.

\bibitem{zhiweizuoling} Zhiwei, Z. and Zuoling, Z. (2001), \textit{The
Hausdorff centred measure of the symmetry Cantor sets,} Approx. Theory and
its Appl., 18, num 2, 49-57.

\bibitem{zhou} Zhou, Z (1998)\textit{\ The Hausdorff measure of the
self-similar sets---the Koch curve}. Sci. China Ser. A 41 , no. 7, 723--728.

\bibitem{zhoufeng} Zhou, Z. and Feng, L. (2000), \textit{A new estimate of
the Hausdorff measure of the Sierpinski Gasket, }Nonlinearity 13, 479-491.

\bibitem{zhufeng1} Zhou, Z. and Feng, L. (2004) \textit{Twelve open problems
on the exact value of the Hausdorff measure and on topological entropy: a
brief survey of recent results}. Nonlinearity 17, no. 2, 493--502.

\bibitem{zhuzhou} Zhu, Z. and Zhou, Z. (2002) \textit{The Hausdorff centred
measure of the symmetry Cantor sets.} Approx. Theory Appl. (N.S.) 18 , no.
2, 49--57.

\bibitem{zhuzhou0} Zhu, Z. and Zhou, Z. (2008) \textit{The exact measures of
a class of self-similar sets on the plane.} Anal. Theory Appl. 24 , no. 2,
160--182.

\bibitem{zhuzhou1} Zhu, Z. and Zhou, Z. (2008) \textit{The centered covering
measures of a class of self-similar sets on the plane. }Real Anal. Exchange
33, no. 1, 215--231.
\end{thebibliography}
\end{document}